\documentclass[10pt]{article}
\usepackage{amsmath,amssymb,amsthm, epsfig}
\usepackage{amsfonts}
\usepackage{amsmath}
\usepackage{amssymb}
\usepackage{color}

\usepackage{hyperref}

\usepackage[mathscr]{eucal}

\title{\bf Nonlinear elliptic equations \\ with high order singularities}

\author{ \textsc{Eduardo V. Teixeira} \\ \text{\footnotesize Universidade Federal do Cear\'a}  \\ \text{\footnotesize Fortaleza, CE, Brazil} }

\date{}

\def \dist {\mathrm{dist}}

\def \Z {\mathscr{Z}}

\newtheorem{theorem}{Theorem}
\newtheorem{lemma}[theorem]{Lemma}

\newtheorem{corollary}[theorem]{Corollary}

\theoremstyle{definition}
\newtheorem{definition}{Definition}

\theoremstyle{remark}

\numberwithin{equation}{section}

\newcommand{\intav}[1]{\mathchoice {\mathop{\vrule width 6pt height 3 pt depth  -2.5pt
\kern -8pt \intop}\nolimits_{\kern -6pt#1}} {\mathop{\vrule width
5pt height 3  pt depth -2.6pt \kern -6pt \intop}\nolimits_{#1}}
{\mathop{\vrule width 5pt height 3 pt depth -2.6pt \kern -6pt
\intop}\nolimits_{#1}} {\mathop{\vrule width 5pt height 3 pt depth
-2.6pt \kern -6pt \intop}\nolimits_{#1}}}

\begin{document}
\maketitle

\begin{abstract}
We study non-variational degenerate elliptic equations with high order singular structures. No boundary data are imposed and singularities occur along an {\it a priori} unknown interior region. We prove that positive solutions have a universal modulus of continuity that does not depend on their infimum value. We further obtain sharp, quantitative regularity estimates for non-negative limiting solutions. 
\smallskip


\smallskip


{
}

\end{abstract}

\section{Introduction} \label{Sct intro}

The theory of singular elliptic equations and its innate backgrounds arise in several areas of pure and applied analysis: theory of fluid mechanics, superconductivity, dynamics of thin films, quenching phenomena, microelectromechanical type of systems,  rupture problems, geometric measure theory, calculus of variations, differential geometry, free boundary theory,  etc. Many of those problems can be modeled in a general framework:
\begin{equation} \label{intro eq01}
	\mathfrak{L}(X,u,Du, D^2 u) = f(X,u),  \quad \mathcal{O} \subset \mathbb{R}^n,
\end{equation}
where $\mathfrak{L}$ is elliptic with respect to the Hessian argument; however ellipticity degenerates along the some {\it a priori} unknown region.  Physical interpretations of the models usually require sign conditions upon the forcing term $f(X, u)$ and impose sign constraint on existing solution $u$, say, $u> 0$; nonetheless, the key estimates are the ones that do not depend upon lower bounds for its infimum.

\par
\medskip

Regularity issues related to such a class of equations are central problems in the modern study of non-linear PDEs. The mathematical analysis of positive solutions to linear isotropic equations with prescribed singular set:
\begin{equation}\label{CRT}
	 \Delta u \sim \dfrac{1}{u^p}, \text{ in } \Omega, \quad u\Big |_{\partial \Omega} = 0,	
\end{equation}
has promoted important development in the theory of 2nd order elliptic equations in the last forty years or so. A very large literature dealing with such a class of problems has evolved from the pioneering work \cite{crt}. In such eruditions, though, singularities occur along a prescribed, fixed, smooth boundary, $\partial \mathcal{O}$. Obtaining appropriate regularity estimates for solutions to singular PDEs with no geometric or smoothness {\it a priori} knowledge on the singular set has been a primary problem since then --- this is the first main issue in the analysis of free boundary type of problems.

\par
\medskip

Isotropic PDEs with free singular sets often arise from critical point theory of non-differentiable functionals. Namely, given a real parameter $t$, one can formally look at energy functionals of the form
$$
	\mathscr{J}_t(u) := \int \left ( |\nabla u|^2 + \lambda(\zeta)u^{t} \chi_{\{u>0\}} \right ) \ d\zeta,
$$
where $\lambda$ is a bounded function.  When $t> 1$, the functional is differentiable and usual methods in the Calculus of Variations can be employed. The case $t = 1$ refers to the well known obstacle problem, see \cite{C, CK} among many other works on this subject. For $0< t <1$, the functional $\mathscr{J}_t$ is continuous but non-differentiable, see for instance \cite{AP}.  Within this range, solutions are continuously differentiable along the interface $\{u=0\}$. The borderline case $t = 0$ bears upon the cavitation problem, \cite{AC}. 

\par 
\medskip

In this paper we treat non-variational elliptic equations with high order singularities of the type
\begin{equation} \label{intro eq03}
	 {u}^\alpha |Du|^\beta \Delta u  = f(X) \cdot  \chi_{\{u>0\}},
\end{equation}
for $\beta \ge 0$ and $-1-\beta < \alpha < 1$, in the sense  of approximating limiting solutions, to be better explained when time comes. For now, it is enlightening to notice that singularities take place along the zero set $\mathscr{Z}(u)$ as well as along the set of critical points $\mathscr{C}(u)$,  which are {\it a priori} unknown sets. No smoothness or geometric information are previously granted upon them. 

The theory of elliptic equations for which ellipticity degenerates solely at the set of critical points, i.e., $\alpha =0$ in \eqref{intro eq01}, has experienced an impressive progress through this past decade, see \cite{BD1, DFQ, IS, ART}  among several other works on this subject. Quite recently, Imbert and Silveste established an important breakthrough in the field. They manage to show that Harnack inequality is still valid when the equation is assumed to be elliptic only when the gradient is large, \cite{IS2}. 

The problem involving singularities upon the zero order term, the case $\alpha \not = 0$, though, is more delicate. While for the variational theory, i.e. the study of solutions coming from a minimization problems, there have been some recent advances,  the appropriate non-variational theory seems much less developed. Non-variational solutions to \eqref{intro eq03} ought to be built up through a limiting process out from positive solutions, or alternatively, by solutions of non-singular approximating equations. To support such an approach, though, one needs  a compactness theorem for positive solutions that is independent of the infimum value. Hence, as to unlock the study of non-variational solutions to singular elliptic equations one initially needs to obtain a universal modulus of continuity  for weak solutions to elliptic partial differential equations with high order singularities. This is the first main result we prove in this article.

\par
\medskip 

A positive function $u$ satisfying Equation  \eqref{intro eq03} in the viscosity sense is of class $C^{1,\frac{1}{1+\beta}}$, see \cite{ART, IS}. Of course such an estimate blows-up as $\inf u \to 0$.  The ultimate goal of this current paper is to prove  that nonnegative limiting solutions to  \eqref{intro eq03} are $C^{\gamma}$ smooth, along their singular set, for the sharp value $\gamma$ given by the algebraic relation
\begin{equation}\label{def gamma}
	\gamma  := \frac{2+\beta}{1+\beta+\alpha}. 
\end{equation}
As in the free boundary theory, the optimal regularity estimate plays a decisive role in the geometric understanding of solutions near their singular points. 

\par
\medskip 

We conclude this Introduction by mentioning that the methods designed for the proof of  the main Theorems presented in this work are, in their very own nature, of non-linear character. The very same universal continuity property as well as the sharp regularity estimate hold true, with essentially the same proofs, if the Laplacian is replaced by a uniform elliptic fully non-linear operator  $F(x,D^2u)$ with ``continuous coefficients". Of course, in this case, the universal estimates depend also upon the ellipticity constants of $F$ and the modulus of continuity of the coefficients. We have chosen to work on the simpler case, as to highlight the novelties and main ideas herein designed for the proofs of such results.

\section{Hypotheses and main results}

We start off this section by commenting on the sign assumption assumed on the forcing term $f(X)$. Clearly if $f\le 0$, positive solutions are superharmonic and then the limiting free boundary is empty. The interesting case, even from applied viewpoint, is when we assume the existence of a constant $c_0>0$ such that
\begin{equation} \label{sign cond}
	c_0 \le  f(X)  \le c_0^{-1}.
\end{equation}
Such a condition will be enforced hereafter in this paper. In addition, we shall also assume uniform continuity of $f$. Recall that a modulus of continuity is a nondecreasing function $\sigma \colon (0, \infty) \to (0,\infty)$ satisfying $\sigma(0^{+}) = 0$. A function is said to be $\sigma$-continuous in $\Omega \subset \mathbb{R}^n$, if
$$
	\left | f(X) - f(Y) \right | \le \sigma \left ( \left |X-Y \right | \right ).
$$

\medskip 

Anisotropic equations as \eqref{intro eq03} are genuinely non-variational, even when diffusion is ruled by the Laplacian. Hence, the notion of solutions cannot be based upon the language of measures or distributions. Rather, the appropriate notion of approximating solutions rests upon the method of building up {\it physical} nonnegative solutions as the limit of positive ones. This is done by means of uniform estimates that do not depend upon the infimum of the solutions. 
We comment that, alternatively, one could consider nonsingular approximating equations, 
\begin{equation} \label{eq_approx}
			\zeta_{\alpha, \epsilon}({u} ) \cdot |Du|^{\beta} \cdot \Delta u = f(X), 
\end{equation}
where $\zeta_{\alpha, \epsilon}(t)\to t^\alpha$ in $(0, \infty)$, and obtain estimates that do not deteriorate as the smoothing parameter $\epsilon$ tends to zero.

Either way, a universal compactness result paves the way for  the theory of limiting solutions. This is our first main goal in this current article.
 
\begin{definition} \label{def LS}  A nonnegative function $u \in C(\mathcal{O})$ is said to be a limiting solution to the non-variational singular equation
	\begin{equation}\label{eq def LS}
		u_0^\alpha \cdot |Du_0|^\beta \cdot \Delta u_0 = f(X) \chi_{\{u_0 > 0 \}},
	\end{equation}
if there exists a sequence of positive functions $u_{j}$, satisfying  
$$
	u_j^\alpha \cdot |Du_j|^\beta \cdot \Delta u_j = f(X)
$$
in the viscosity sense, converging locally uniformly to $u_0$ in $\mathcal{O}$.
\end{definition}

Before presenting the main results proven in this work, let us declare that any constant or entity that depends only on dimension, $\|u\|_\infty$, $\alpha$ and $\beta$, $c_0$ and the modulus of continuity of $f$ will be called {\it universal}. We now pass to discuss the central theorems to be delivered in this manuscript. The first key result we prove is a universal compactness theorem for positive solutions to singular equations.

\begin{theorem}\label{main_prop} Let $v$ be a bounded positive viscosity solution to 
\begin{equation}\label{Eq sct aux}
	u^\alpha |Du|^\beta \Delta u = f(X),
\end{equation}
in $\mathcal{O}\subset \mathbb{R}^n$, with $\beta \ge 0$, $1>\alpha > -(1+\beta)$ and $f$ satisfying \eqref{sign cond}. Given a subdomain $\mathcal{O}' \Subset \mathcal{O}$, there exists a modulus of continuity $\varpi$, depending only on universal parameters and $\mathcal{O}' $, such that $u$ is $\varpi$-continuous in $\mathcal{O}'$. 
\end{theorem}
\medskip 

The main tool used in the proof of Compactness Theorem \ref{main_prop} is the so called Ishii-Lions method, see for instance \cite{IL, BCI, AT}. We will carry it out in Section \ref{sct aux}. Of particular interest, Theorem \ref{main_prop} allows us to develop the theory of limiting solutions to singular elliptic equations with free boundaries
\begin{equation}\label{eq fb}
	u_0^\alpha|Du_0|^\beta \Delta u_0 = f(X) \chi_{\{u_0>0 \}}, \quad \text{ in } \mathcal{O} \subset \mathbb{R}^n,
\end{equation}
as forecasted by Definition \ref{def LS}. 

\medskip 

Let us know turn our attention to the regularity theory for limiting solutions of the free boundary problem \eqref{eq fb}. If follows from the Theorem proven in \cite{ART, IS} that $u_0$ is locally of class $C^{1,\frac{1}{1+\beta}}$ within its positive set. The major, key issue though is to understand the optimal growth behavior of such a function along the free boundary $\partial \{u_0 > 0 \}$. Such a sharp estimate is given by the next result.

\begin{theorem}\label{MAIN}   Let $u_0$ be a bounded limiting solution to   \eqref{eq fb} with $\beta \ge 0$, $1>\alpha > -(1+\beta)$ and $f$ satisfying \eqref{sign cond}. Fixed a subdomain $\mathcal{O}' \Subset \mathcal{O}$,  there exists a constant $C\ge 1$, depending only upon universal parameters and  $\dist (\partial \mathcal{O}', \partial \mathcal{O})$, such that if $Z \in \partial \{u_0> 0 \} \cap \mathcal{O}'$, then 
$$
	\sup\limits_{B_r(Z)} u_0 \le C r^{\frac{2+\beta}{1+\beta+\alpha}},
$$
for any $r < \dist (Z, \partial \mathcal{O}')$.
\end{theorem}

A direct consequence of Theorem \ref{MAIN} is the sharp control of the  value of $u$ at a point off the free boundary in terms of its distance to the zero set $\Z(u) := \{u_0 = 0\}$.

\begin{corollary}\label{Cor} Let $u_0$ be as in the statement of Theorem \ref{MAIN}. Then for any point $X \in \{u_0 > 0 \}$ with $\dist(X, \Z(u))< \dist(X, \partial \mathcal{O})$, there holds
$$
	u_0(X) \le C \cdot \dist(X, \Z(u_0))^{\frac{2+\beta}{1+\beta+\alpha}}.
$$
\end{corollary}

\bigskip

It is also interesting to observe that positive solutions to the critical equation $u^{-(1+\beta)} |\Delta u|^\beta \Delta u = f(x)$ is universally bounded by below. A proof of this fact is based on a barrier argument, and we omit here.

We conclude this Section by commenting that nonnegative solutions obtained as limit of minimal solutions are non-degenerate, i.e., 
$$
	\sup\limits_{B_r(Z)} u_0  \sim r^{\frac{2+\beta}{1+\beta+\alpha}},
$$
 for any free boundary point $Z\in \partial \{u_0 > 0 \}.$ The proof of this fact follows as in \cite[Section 8]{ART}, see also \cite{AT, RT}. At least for non-degenerate limiting solutions, when $\alpha \ge 0$, a combination of Theorem \ref{MAIN} and \cite[Corollary 3.2]{ART}, gives that $u_0$ is precisely $C^{\gamma}$ up to the free boundary. That is, $C^\gamma$ regularity holds locally around the free boundary. 

Indeed, in such a scenario, for $d:= \dist (Z, \partial \{u_0>0\}) \ll 1$, we initially estimate from \cite[Corollary 3.2]{ART} (see also \cite{IS})
\begin{equation}\label{final rmk sct3 01}
	[u_0]_{C^\gamma(B_{d/4}(Z))} \lesssim \frac{1}{d^{\gamma}} \left ( \|u_0\|_{L^{\infty}(B_{d/2}(Z))} +  d^{\frac{2+\beta}{1+\beta}} \cdot  \|u_0^{-\alpha}\|^{\frac{1}{1+\beta}}_{L^{\infty}(B_{d/2}(Z))} \right ).
\end{equation}
Applying Theorem \ref{MAIN} and non-degeneracy respectively, we estimate
\begin{equation}\label{final rmk sct3 02}
	 \|u_0\|_{L^{\infty}(B_{d/2}(Z))} \lesssim d^\gamma \quad \text{ and } \quad  \|u_0^{-\alpha}\|^{\frac{1}{1+\beta}}_{L^{\infty}(B_{d/2}(Z))}  \lesssim d^{\frac{-\alpha \gamma}{1+\beta}}.
\end{equation}
Finally, plugging \eqref{final rmk sct3 02} into \eqref{final rmk sct3 01}, and taking into account the precise value of $\gamma$, \eqref{def gamma}, we verify that the $C^{\gamma}$ norm of non-degenerate limiting solutions is under control up the the free boundary. 

For improved estimates that hold exclusively along  the free boundaries, see \cite{PT, T1, T2}.

\section{Universal Compactness} \label{sct aux}

This Section is devoted to the proof of Theorem \ref{main_prop}, where it is established the key universal compactness property for solutions, independent of the infimum of $u$.

\bigskip

	We restrict the proof to the case $\alpha > 0$, as the other range follows from \cite{IS, ART}, once one multiplies the equation by $u^{-\alpha}$. Define $v := u^{1/\gamma}$, where $\gamma$ is given by \eqref{def gamma}. We will show that $v$ is uniformly continuous independent of its infimum value. Fixed $X_0 \in \mathcal{O}$, let us denote by $d := \dist (X_0, \partial \mathcal{O})$; if $ \dist (X_0, \partial \mathcal{O})=+\infty,$ then we fix $d$ to be any arbitrary large (but finite) number. With no loss of generality, we will assume $X_0 = 0$. We will show that for any $\delta>0$ given, there exist $\varrho_\delta \ll 1$ and $\mathrm{L}_\delta \gg 1$ such that
	\begin{equation}\label{propeq0}
			\sup\limits_{  \bar{\Omega} \times \bar{\Omega}} \left \{v(\varrho_\delta X) - v(\varrho_\delta Y) - \mathrm{L}_\delta \omega (|X-Y|) - \kappa \cdot \left (|X|^2 + |Y|^2 \right )\right \} \le \delta,
\end{equation}
	where $\omega(\zeta) = \zeta - \frac{1}{10\sqrt{d}} \zeta^{3/2}$, for $\zeta \le d$, $\omega(\zeta) =  \frac{9}{10}d$ for $\zeta \ge d$.  The parameter $\kappa >0$ is chosen such that
	\begin{equation} \label{propeq_kappa}
		 \dfrac{8\|v\|_\infty}{d^2} \le \kappa.
	\end{equation}
	Notice that in the case $d =+\infty$, then  $\kappa>0$ can be chosen arbitrarily small. 
	
	In order to verify \eqref{propeq0}, let us suppose, for the sake of contradiction, that
	\begin{equation}\label{Phi}
		\sup\limits_{ \bar{\Omega} \times \bar{\Omega}} \left \{v(\varrho_\delta X) - v(\varrho_\delta Y) - \mathrm{L}_\delta \omega (|X-Y|) - \kappa \cdot \left (|X|^2 + |Y|^2 \right )\right \} > \delta.
	\end{equation}
	We shall reach an inconsistency by selecting $ \mathrm{L}_\delta$ big enough and $\varrho_\delta$ tiny. Let $(\bar{X}, \bar{Y})$ denote a pair of points where such a maximum is attained. It  follows from the thesis that
	\begin{equation}\label{propeq1}
		 \kappa \cdot  (|\bar{X}|^2 + |\bar{Y}|^2 ) < 2 \|v\|_\infty,
	\end{equation}
 	thus, $\bar{X}$ and $\bar{Y}$ are interior points and $|\bar{X}- \bar{Y}| < d$. Clearly $\bar{X} \not = \bar{Y}$, otherwise we obtain a direct contradiction on \eqref{Phi}. Define in the sequel the vectors
	\begin{eqnarray} 
		\xi_X := \mathrm{L}_\delta \omega'(|\bar{X} -\bar{Y}|)\eta +  2 \kappa  \bar{X}	\label{propeq2}\\
		\xi_Y := \mathrm{L}_\delta \omega'(|\bar{X} -\bar{Y}|)\eta -  2 \kappa \bar{Y}	\label{propeq3},
	\end{eqnarray}
	where $\eta := \frac{\bar{X} -\bar{Y}}{|\bar{X} -\bar{Y}|}$. Notice that 
	\begin{equation}\label{propeq3.01}
		\omega'(|\bar{X} - \bar{Y}|) \ge \omega'(r) = \frac{17}{20}.
	\end{equation}
	Also, 
	\begin{equation}\label{propeq3.02}
		2\kappa \cdot \max \left \{ \left | \bar{X} \right |,  \left |  \bar{Y}\right | \right \} < \dfrac{8\|v\|_\infty}{d},
	\end{equation}
	thus
	\begin{eqnarray} 
		\xi_X &\approx & \sigma_0 \mathrm{L}_\delta +  \text{O}(\frac{\|v\|_\infty}{d})	\label{propeq2.5}\\
		\xi_Y  & \approx &   \sigma_0 \mathrm{L}_\delta +  \text{O}(\frac{\|v\|_\infty}{d})	\label{propeq3.5},
	\end{eqnarray}
	for a nonzero vector $\sigma_0$. 
	From Jensen-Ishii's approximation Lemma, see \cite[Theorem 3.2]{UserG},  for $i > 0$ small enough, it is possible to find matrices $M_X$ and $M_Y$ with 
	\begin{eqnarray}
		(\xi_X, M_X)& \in &\mathcal{J}^{-}(v, \bar{X}), \label{propeq2.1}\\
		(\xi_Y, M_Y) &  \in & \mathcal{J}^{+}(v, \bar{Y}),  \label{propeq2.2}
	\end{eqnarray} 
where $\mathcal{J}^{-}$ and $\mathcal{J}^{+}$ denote the subjet and superjet respectively, verifying the following matrix inequality
	\begin{equation}\label{propeq4}
		\begin{pmatrix}
			M_X & 0 \\
			0 & -M_Y 
\end{pmatrix}
\leq 
\begin{pmatrix}
		Z  & -Z \\
		-Z & Z
\end{pmatrix} + ( 2\kappa +i) \cdot \text{Id}_{2n\times 2n} , 	
\end{equation}	
where 
\begin{equation}\label{propeq5}
	\begin{array}{lll}
		Z&=& L_\delta \omega''(|\bar{X} -\bar{Y}|) \dfrac{ (\bar{X} -\bar{Y}) \otimes (\bar{X} -\bar{Y}) }{|\bar{X} -\bar{Y}|^2} + \\ 
		&+& \dfrac{\omega'(|\bar{X} -\bar{Y}|)}{|\bar{X} -\bar{Y}|} \left \{ \text{Id}_{n\times n}  - \dfrac{ (\bar{X} -\bar{Y}) \otimes (\bar{X} -\bar{Y}) }{|\bar{X} -\bar{Y}|^2}  \right \}.
	\end{array}
\end{equation}
Applying inequality \eqref{propeq4} to vectors of the form $(\xi, \xi)$, we conclude 
\begin{equation}\label{propeq6}
		\text{Spect}(M_Y - M_X) \in (- 4\kappa - i , +\infty).
\end{equation}
However, if we apply to the special vector $(\eta, -\eta)$, we conclude 
\begin{equation}\label{propeq7}
		\text{Spect}(M_Y - M_X) \cap (\frac{c}{\sqrt{d}} \mathrm{L}_\delta - 4\kappa - i,  +\infty) \not = \emptyset,
\end{equation}
for a universal number $c>0$, depending only upon our choice for $\omega$. Combining \eqref{propeq6} and \eqref{propeq7}, we end up with
\begin{equation}\label{propeq7.5}
		\left ( \frac{c}{\sqrt{d}} \mathrm{L}_\delta - n(4\kappa + i) \right ) + \text{Trace} (M_X)  < \text{Trace} (M_Y).
\end{equation}
Notice that if we choose  
\begin{equation}\label{propeq7.2}
	\mathrm{L}_\delta \gtrsim  \|v\|_\infty d^{-3/2} + i \sqrt{d} ,
\end{equation}
then the term $\frac{c}{\sqrt{d}} \mathrm{L}_\delta - n(4\kappa + i) $ becomes positive.  Also, since from our thesis
\begin{equation} \label{propeq16.00}
	v(\varrho_\delta \bar{X}) - v(\varrho_\delta \bar{Y}) > \delta,
\end{equation}
we have the upper control
\begin{equation}  \label{propeq16.1} 
	   v^{-1}(\varrho_\delta \bar{X}) <  \dfrac{1}{\delta}.
\end{equation}
It also follows directly from \eqref{propeq16.00} the lower control
\begin{equation} \label{propeq16.2} 
	v^{-1}(\varrho_\delta \bar{Y})  >  \frac{\delta}{\|v\|_\infty^2} + v^{-1}(\varrho_\delta \bar{X}).
\end{equation}
 
Next, due to the sharp choice of $\gamma$,  the function $v$ satisfies
$$
	\Delta v(X) =\left ( (1-\gamma) |\nabla v(X)|^2 + \dfrac{1}{\gamma^{1+\beta}} |\nabla v(X)|^{-\beta} f(X) \right )  v^{-1}(X),
$$
and hence $v(\varrho_\delta X)$ verifies the equation
\begin{equation}\label{eq scaled}
	\Delta \phi - (1-\gamma) |\nabla \phi|^2 \phi^{-1} =  \varrho_\delta^{2+\beta} \dfrac{1}{\gamma^{1+\beta}} |\nabla \phi (X)|^{-\beta} f(\varrho_\delta X) \phi^{-1}(X).
\end{equation}
In the sequel, we confront \eqref{eq scaled} with \eqref{propeq2.1} and \eqref{propeq2.2} as to write up the following pointwise inequalities 
\begin{eqnarray}
	\text{Trace} (M_X)   &\ge&  \label{propeq8} \left (  (1-\gamma) |\xi_X|^2 +  \dfrac{ \varrho_\delta^{2+\beta}}{\gamma^{1+\beta}} |\xi_X|^{-\beta} f(\varrho_\delta \bar{X}) \right )  v^{-1}(\varrho_\delta \bar{X})\\
	\text{Trace} (M_Y)   &\le&  \label{propeq9} \left (  (1-\gamma) |\xi_Y|^2 +  \dfrac{ \varrho_\delta^{2+\beta}}{\gamma^{1+\beta}} |\xi_Y|^{-\beta} f(\varrho_\delta \bar{Y}) \right )  v^{-1}(\varrho_\delta \bar{Y}).
\end{eqnarray}
Combining \eqref{propeq8}, \eqref{propeq9}  and  \eqref{propeq7.5}, taking into account the choice \eqref{propeq7.2}, we obtain
\begin{equation} \label{propeq17}
	 \left  ( \dfrac{(\gamma-1)|\xi_Y|^2 -  \frac{ \varrho_\delta^{2+\beta}}{\gamma^{1+\beta}} |\xi_Y|^{-\beta} f(\varrho_\delta \bar{Y})}{(\gamma-1) |\xi_X|^2 -  \frac{ \varrho_\delta^{2+\beta}}{\gamma^{1+\beta}} |\xi_X|^{-\beta} f(\varrho_\delta \bar{X})}  \right ) v^{-1}(\varrho_\delta \bar{Y}) \le v^{-1}(\varrho_\delta \bar{X}),
\end{equation}
provided $L_\delta \gg 1$ is chosen even bigger, if necessary. In view of \eqref{propeq2.5} and \eqref{propeq3.5}, given  $\iota \ll 1$, to be chosen {\it a posteriori} depending only on $\|v\|_\infty$ and $\delta$, for 
$\mathrm{L}_\delta \gg 1,$ and $\varrho_\delta \ll 1$ 
depending on the modulus of continuity of $f$ and $\iota$ -- but choices independent of the infimum of $v$ -- there holds
\begin{equation} \label{propeq17.2}
	\dfrac{(\gamma-1)|\xi_X|^2 -  \frac{ \varrho_\delta^{2+\beta}}{\gamma^{1+\beta}} |\xi_X|^{-\beta} f(\varrho_\delta \bar{X})}{(\gamma-1) |\xi_Y|^2 -  \frac{ \varrho_\delta^{2+\beta}}{\gamma^{1+\beta}} |\xi_Y|^{-\beta} f(\varrho_\delta \bar{Y})} \le (1+\iota).
\end{equation}
Thus, from \eqref{propeq17}, taking i`ount \eqref{propeq16.1} we find
\begin{equation} \label{propeq17.3}
	 v^{-1} (\varrho_\delta  \bar{Y}) \le  \iota \dfrac{1}{\delta} + v^{-1}(\varrho_\delta  \bar{X}).
\end{equation}
In the sequel, we select
\begin{equation} \label{propeq18.1}
 	\iota \le  \frac{\delta^2}{10\|v\|_\infty^2},
\end{equation}
which gives a  contradiction on  \eqref{propeq16.2}. We have shown that, for any pair of interior points, $X, Y \in \mathcal{O}' \Subset \mathcal{O}$, and any positive number $\delta > 0$, there holds
\begin{equation} \label{propeq20}
	|v(X) - v(Y)| \le \varrho_\delta^{-1}\max \{\mathrm{L}_\delta, \tau, \frac{\|v\|_\infty^4}{d\delta^4}\} \cdot |X-Y| + 2 \kappa  \varrho_\delta^{-2} |X-Y|^2 + \delta,
\end{equation}
where $\tau$ accounts the choices made at \eqref{propeq7.2} and \eqref{propeq17.2}. In particular, all choices are independent of infimum of $v$. It is easy to check that \eqref{propeq20} gives a universal modulus of continuity for $v$. The proof of Theorem \ref{main_prop} follows. \hfil $\square$

\bigskip

We conclude this Section by commenting on the universal modulus of continuity $\varpi$ found in Theorem \ref{main_prop}. As inspection on \eqref{propeq20} reveals the existence of a strictly decreasing function $\Omega \colon (0, 1) \to (0,\infty)$, such that, for any $0< \delta< 1$, there holds
\begin{equation} \label{propeq21}
	|v(X) - v(Y)| \le \Omega(\delta)\cdot |X-Y| + \delta,
\end{equation}
provided $|X-Y| \le 1$. Indeed, $\Omega(\delta) \approx \varrho_\delta^{-1} L_\delta$, with $\varrho_\delta$ depending only on universal parameters and on the modulus of continuity of $f$ and $L_\delta$ depending only on universal numbers. Let us define the (increasing) function $\Xi(\delta) := \delta^{-1} \Omega(\delta)$ and in the sequel we set 
\begin{equation} \label{propeq22}
	\varpi(t) := \Xi^{-1}(\frac{1}{t}),
\end{equation}
which is a modulus of continuity. From the very definition of $\varpi(t)$, it follows that
\begin{equation} \label{propeq23}
	\frac{1}{\varpi(t)}\Omega(\varpi(t)) =  \Xi(\varpi(t)) = \frac{1}{t}.
\end{equation}
Now, given two points $X, Y$ with $0<|X-Y| \le 1$, select in \eqref{propeq21} $\delta = \varpi(|X-Y|)$, taking into account \eqref{propeq23}, and estimate:
\begin{equation} \label{propeq24}
	\begin{array}{lll}
		|v(X) - v(Y)| &\le& \Omega(\varpi(|X-Y|))\cdot |X-Y| +  \varpi(|X-Y|) \\
		&\le & 2\varpi(|X-Y|).
	\end{array}
\end{equation}
Hence $2\varpi$ is a modulus of continuity for $v$, which is universal. 

Let us further mention that when $\Omega(\delta) \approx \delta^{-M}$ -- which is the case when $f$ is H\"older continuous -- then Theorem  \ref{main_prop} provides a universal $C^{0,\frac{1}{M+1}}$--H\"older continuity estimate for $v$. 


\section{Regularity along the free boundary}

In this Section we aim to prove that limiting solutions to \eqref{eq fb} are of class $C^{\gamma}$ along the free boundary $\{u_0 > 0\}$. The strategy will be based on a flatness improvement device which allows us to control the growth of $u_0$ in proper geometric dyadic balls. Initially we need a Lemma.

\begin{lemma}\label{Flat Lemma1} Let $u\colon B_1 \to \mathbb{R}$ be a positive solution to 
	\begin{equation}\label{Eq Opt Est}	
		u^\alpha |D u|^\beta  \Delta u = f(X), \quad \text{in } B_1,
	\end{equation}
	satisfying $|u| \le 1$. Given a positive number $\theta > 0$, there exists $\eta = \eta(\theta) > 0$ depending only on $\theta$ and the universal parameters, such that if 
$$
	\inf\limits_{B_{1/2}} u \le \eta \quad \& \quad \|f\|_{L^\infty(B_1)} \le \eta^{1+\alpha+\beta},
$$
then
$$
	\sup\limits_{B_{1/4}} u \le \theta.
$$
\end{lemma}
\begin{proof}
	Let us suppose, for the sake of contradiction, that the thesis of the Lemma fails to hold. That is, there exists a sequence of positive functions $u_j \colon B_1 \to \mathbb{R}$, satisfying,
	$$
		|u_j| \le 1; \quad \iota_j := \inf\limits_{B_{1/2}} u_j = \text{o}(1), \quad \left \| f_j(X) \right \|_{L^\infty(B_1)} \le \iota_j^{1+\alpha+\beta},
	$$
	where $f_j(X) =: |\nabla u_j|^\beta u_j^\alpha \Delta u_j$ in the viscosity sense and $\theta_0 > 0$, such that 
	$$
		\sup\limits_{B_{1/4}} u_j \ge \theta_0.
	$$
	In the sequel, we define the normalized function
	$$
		v_j(X) := \dfrac{u_j(X)}{\iota_j}.
	$$
	It is clear that $\inf v_j =1$. Direct computations yield
	$$
		|D v_j|^\beta \Delta v_j = \dfrac{v_j^{-\alpha}}{\iota_j^{1+\alpha+\beta}} \cdot f_{j}(X) =: \mathfrak{g}_j(X),
	$$
	which is a bounded function in $B_{1/2}$. Now, applying Harnack inequality, see for instance \cite{IS2}, we deduce
	$$
		C \ge \sup_{B_{1/4}} v_j (X) \ge \dfrac{\sup\limits_{B_{1/4}} u_j (X)}{\iota_j},
	$$
	which gives a contradiction if $\iota_j \ll 1$. 
\end{proof}

\medskip

We now proceed into the proof of Theorem \ref{MAIN}. Hereafter, let us label:
\begin{equation}\label{eta star}
	\eta_\star := \eta(4^{-\gamma});
\end{equation}
i.e., $\eta_\star$ is the positive number given by Lemma \ref{Flat Lemma1}, when one takes $\theta =4^{-\gamma}$. Recall 
$$
	\gamma = \dfrac{2+\beta}{1+\alpha+\beta},
$$
gives the aimed optimal regularity.

Our initial observation is that if $u$ is a positive solution to Equation \eqref{Eq Opt Est}, then the scaled function $v(X) := u(\varrho X)$, with 
$$
	\varrho := \eta_\star^{1/\gamma} \sqrt[2+\beta]{\|f\|_\infty^{-1}}
$$
solves the same equation in $B_1$ with the right hand side bounded by $ \eta_\star^{1+\alpha+\beta}$.
Hence, from now on, we can assume, modulo a fixed zoom-in, that the RHS of  Equation \eqref{Eq Opt Est} is bounded by $ \eta_\star^{1+\alpha+\beta}$.

\medskip

Now, let $u_0$ be a nonnegative limiting solution and assume $0 \in \partial \{u_0 > 0 \}$. We want to show that there exists a universal constant $C>0$ such that
$$
	\sup\limits_{B_r} u_0 (X) \le C r^\gamma.
$$
For that, let $u_j$ be a sequence of positive solutions converging locally uniformly to $u_0$. Since $u_j(0) = \text{o}(1)$, there exists $j_0 \in \mathbb{N}$, such that if $j\ge j_0$,
$$
	\inf\limits_{B_{1/2}} u_j \le u_j(0) \le \eta_\star.
$$
It then follows from Lemma \ref{Flat Lemma1} that
\begin{equation}\label{thm_final1}
	\sup\limits_{B_{1/4}} u_j(X) \le 4^{-\gamma}, \quad \forall j \ge j_1.
\end{equation}
In the sequel, we define $v_j^1 \colon B_{1} \to \mathbb{R}$, by
$$
	v_j^1(X) := 4^{\gamma} u_j(\dfrac{1}{4} X).
$$
It follows from \eqref{thm_final1} that $v^1_j$ is a normalized function satisfying
$$
	 \left |  (v_j^1)^\alpha |D v_j^1|^\beta \Delta v^1_j \right | \le \eta_\star^{1+\alpha+\beta},
$$
in the viscosity sense. We can now choose $j_1 > j_0$, such that 
$$
	\inf\limits_{B_{1/2}} v_j^1 \le 4^\gamma u_j(0) \le \eta_\star.
$$
Applying Lemma \ref{Flat Lemma1} to $v^1_j$ and rescaling it back to $u_j$, gives
\begin{equation}\label{thm_final2}
	\sup\limits_{B_{1/16}} u_j(X) \le 16^{-\gamma}, \quad \forall j \ge j_1.
\end{equation}
Continuing the reasoning, we define $v_j^2 \colon B_{1} \to \mathbb{R}$, by
$$
	v_j^2(X) := 4^{\gamma} v^1_j(\dfrac{1}{4} X) = 16^\gamma u_j(\dfrac{1}{16} X).
$$
Again, through scaling arguments, one verifies that
$$
	 \left |  (v_j^2)^\alpha |D v_j^2|^\beta \Delta v^2_j \right | \le \eta_\star^{1+\alpha+\beta},
$$
in the viscosity sense. For another natural number, $j_2 > j_1$, there holds
$$
	\inf\limits_{B_{1/2}} v^2_j \le 16^\gamma \cdot u_j(0) \le \eta_\star,
$$
for all $j\ge j_2$. Applying Lemma \ref{Flat Lemma1} to $v_j^2$ and rescaling the estimate back to $u_j$  gives
\begin{equation}\label{thm_final3}
	\sup\limits_{B_{1/64}} u_j(X) \le 64^{-\gamma}, \quad \forall j \ge j_2.
\end{equation}
Proceeding inductively, we prove that for any natural number $k \ge 1$, there exists $j_k \in \mathbb{N}$ such that
\begin{equation}\label{thm_final4}
	\sup\limits_{B_{4^{-k}}} u_j(X) \le 4^{-k\gamma}, \quad \forall j \ge j_k.
\end{equation}
Finally, given $0 < r < 1/4$, let $k$ be such that $4^{-(k+1)} < r \le 4^{-k}$. We can estimate
$$
	\sup\limits_{B_r} u_0 \le \sup\limits_{B_{4^{-k}}} u_0 \le 4^{-k \gamma} \le 4^\gamma r^\gamma,
$$
and the Theorem is proven.  $\square$



\bigskip

\bibliographystyle{amsplain, amsalpha}

\bigskip

\noindent \textsc{Eduardo V. Teixeira} \\
\noindent Universidade Federal do Cear\'a \\
\noindent Departamento de Matem\'atica \\
\noindent Campus do Pici - Bloco 914, \\
\noindent Fortaleza, CE - Brazil 60.455-760 \\
\noindent \texttt{teixeira@mat.ufc.br}

\end{document}